\documentclass[11pt]{article}

\usepackage[a4paper,margin=1in]{geometry}
\usepackage{amsmath,amssymb,amsthm,mathtools}
\usepackage{mathrsfs}
\usepackage{aliascnt}
\usepackage{hyperref}
\usepackage[nameinlink,capitalize]{cleveref}

\hypersetup{
  colorlinks=true,
  linkcolor=blue,
  citecolor=blue,
  urlcolor=blue
}

\numberwithin{equation}{section}

\newtheorem{theorem}{Theorem}[section]

\newaliascnt{lemma}{theorem}
\newtheorem{lemma}[lemma]{Lemma}
\aliascntresetthe{lemma}

\newaliascnt{corollary}{theorem}

\aliascntresetthe{corollary}

\newaliascnt{proposition}{theorem}

\aliascntresetthe{proposition}

\theoremstyle{remark}
\newaliascnt{remark}{theorem}

\aliascntresetthe{remark}

\crefname{theorem}{Theorem}{Theorems}
\Crefname{theorem}{Theorem}{Theorems}
\crefname{lemma}{Lemma}{Lemmas}
\Crefname{lemma}{Lemma}{Lemmas}
\crefname{corollary}{Corollary}{Corollaries}
\Crefname{corollary}{Corollary}{Corollaries}
\crefname{proposition}{Proposition}{Propositions}
\Crefname{proposition}{Proposition}{Propositions}
\crefname{remark}{Remark}{Remarks}
\Crefname{remark}{Remark}{Remarks}

\newcommand{\Fp}{\mathbb F_p}
\newcommand{\Fpx}{\mathbb F_p^{\times}}
\newcommand{\QR}{\mathcal Q}
\newcommand{\one}{\mathbf 1}
\newcommand{\sgn}{\operatorname{sgn}}
\newcommand{\Pf}{\operatorname{Pf}}
\newcommand{\adj}{\operatorname{adj}}

\title{A Pfaffian Proof and Generalization of a Conjecture of Sun Zhiwei}
\author{Hong-Ge Chen,\qquad Fei Liu}
\date{}

\begin{document}
\maketitle

\begin{abstract}
Let $p$ be an odd prime, let $n=(p-1)/2$, and let $\chi=(\frac{\cdot}{p})$, with $\chi(0)=0$.  For $a\in\mathbb F_p^\times$ define
\[
  D_a(x)=\det_{1\le i,j\le n}(x+\chi(i^2-aj)),
  \qquad
  D_a^{(0)}(x)=\det_{0\le i,j\le n}(x+\chi(i^2-aj)).
\]
We prove
\[
  D_a(0)=0
  \quad\Longleftrightarrow\quad
  p\equiv 3 \pmod 4
  \quad\text{and}\quad
  \chi(a n!)=1.
\]
For $p\equiv3\pmod4$ we also give explicit Pfaffian-square factorizations of $D_a(x)$ and $D_a^{(0)}(x)$.  Let $s_p=(-1)^{\lfloor(p+1)/8\rfloor}$.  If $\chi(a n!)=1$, then $s_pD_a(x)/x=s_pD_a^{(0)}(x)$ is a positive integer square.  If $\chi(a n!)=-1$, then there is a positive integer $\sigma$ such that
\[
  s_pD_a(x)=\sigma^2(nx-1),\qquad
  s_pD_a^{(0)}(x)=-\sigma^2\bigl(n+(2n+1)x\bigr).
\]
The case $a=n!$ settles Sun's Conjecture 4.1.
\end{abstract}



\section{Introduction and main results}\label{sec:introduction}

Throughout the paper $p$ is an odd prime,
$
  n=\frac{p-1}{2},
$
and $\chi$ denotes the Legendre symbol modulo $p$, extended by $\chi(0)=0$.  For $a\in\Fp^\times$ define
\begin{align*}
  D_a(x)
  &=\det_{1\le i,j\le n}\bigl(x+\chi(i^2-aj)\bigr),\\
  D_a^{(0)}(x)
  &=\det_{0\le i,j\le n}\bigl(x+\chi(i^2-aj)\bigr).
\end{align*}
The variable $x$ is an indeterminate.  For $a=n!$, these are the determinants appearing in \cite[Conjecture~4.1]{Sun}.

Sun studied several determinants with Legendre-symbol entries in \cite{Sun}.  Among them is
\[
  W_p=\det_{0\le i,j\le n}\chi(i^2-n!j).
\]
Sun evaluated $\chi(W_p)$ in \cite[Theorem~1.5]{Sun} and conjectured the stronger square statement for the two polynomial deformations above, together with the vanishing criterion for $D_{n!}(0)$.  Related determinant problems have since been treated in \cite{GrinbergSunZhao,Wu,WangWu,RenLuo,RenSun,YangZhang}.  The result here is an exact polynomial factorization, not only a square-class or congruence statement.  We prove Sun's conjecture through an $a$-uniform theorem.

First we record the special case that answers \cite[Conjecture~4.1]{Sun}.

\begin{theorem}[Sun's Conjecture 4.1]\label{thm:sun}
One has
\[
  \det_{1\le i,j\le n}\chi(i^2-n!j)=0
\]
if and only if $p\equiv 3\pmod 4$.  If $p\equiv 3\pmod 4$, then
\[
  s_p\det_{0\le i,j\le n}\bigl(x+\chi(i^2-n!j)\bigr)
\]
and
\[
  s_p\,\frac{1}{x}
  \det_{1\le i,j\le n}\bigl(x+\chi(i^2-n!j)\bigr)
\]
are equal positive integer squares and are independent of $x$.
\end{theorem}

This follows from two uniform statements.

\begin{theorem}[Vanishing criterion]\label{thm:vanishing}
Let $p$ be an odd prime and let $a\in\Fp^\times$.  Then
$
  D_a(0)=0
$
if and only if
\[
  p\equiv 3\pmod 4
  \quad\text{and}\quad
  \chi(a n!)=1.
\]
\end{theorem}

For the square formulas, put
$
  s_p=(-1)^{\lfloor (p+1)/8\rfloor}.
$

\begin{theorem}[Square and linear-square formulas]\label{thm:square}
Assume $p\equiv 3\pmod 4$ and let $a\in\Fp^\times$.

\smallskip
\noindent\textup{(i)} If $\chi(a n!)=1$, then there exists a positive integer $\rho_{p,a}$ such that
\[
  s_p\,\frac{D_a(x)}{x}
  =s_p\,D_a^{(0)}(x)
  =\rho_{p,a}^{\,2}.
\]
In particular, both expressions are independent of $x$.

\smallskip
\noindent\textup{(ii)} If $\chi(a n!)=-1$, then there exists a positive integer $\sigma_{p,a}$ such that
\[
  s_p\,D_a(x)=\sigma_{p,a}^{\,2}(nx-1)
\]
and
\[
  s_p\,D_a^{(0)}(x)
  =-\sigma_{p,a}^{\,2}\bigl(n+(2n+1)x\bigr).
\]
\end{theorem}

We briefly indicate the proof.  The main point is to replace the original determinants by a universal half-system determinant.  When $p\equiv3\pmod4$, the set $\{aj:1\le j\le n\}$ contains one representative from each pair $\{u,-u\}\subset\Fpx$; after the rows are indexed by the quadratic residues, the columns are encoded by a sign function on the quadratic residues.  The determinant without the rank-one term $xJ$ factors as $BE$, where $E=\begin{psmallmatrix}S&0\\ U&I\end{psmallmatrix}$ and $S$ is a principal block of the skew-symmetric matrix $B^{-1}P$.  This is the source of the Pfaffian square.  The row and column sums then control the $xJ$ term, while Zolotarev's lemma gives the remaining global column sign $(-1)^{\lfloor(p+1)/8\rfloor}$.

The sections are arranged to isolate these ingredients.  \Cref{sec:vandermonde} proves a Vandermonde-type congruence used for the vanishing criterion and for nonsingularity checks.  \Cref{sec:half-system} proves the Pfaffian factorization for arbitrary half-systems.  \Cref{sec:dilated} identifies the half-system attached to the dilation by $a$ and computes the column-permutation sign.  Finally, \Cref{sec:proofs} combines these ingredients to prove \cref{thm:vanishing,thm:square}, and hence \cref{thm:sun}.

\section{A Vandermonde congruence}\label{sec:vandermonde}

We shall use the elementary congruence
\[
  \chi(z)\equiv z^n\pmod p
  \qquad(z\in\Fp),
\]
where both sides are interpreted in $\Fp$.

\begin{lemma}\label{lem:vandermonde}
Let $y_1,\dots,y_n$ be distinct nonzero elements of $\Fp$.  Then, in $\Fp[x]$,
\[
  \det_{1\le i,j\le n}\bigl(x+\chi(i^2+y_j)\bigr)
  = C(y_1,\dots,y_n)
  \left(x+1+(-1)^{n-1}\prod_{j=1}^n y_j\right),
\]
where $C(y_1,\dots,y_n)\in\Fp^\times$.
\end{lemma}

\begin{proof}
Put $q_i=i^2$.  The elements $q_1,\dots,q_n$ are the nonzero quadratic residues and are distinct.  Modulo $p$,
\[
  x+\chi(q_i+y_j)
  \equiv x+(q_i+y_j)^n.
\]
Since $q_i^n=1$, we have
\[
  x+(q_i+y_j)^n
  =x+1+y_j^n+
  \sum_{r=1}^{n-1}\binom nr q_i^r y_j^{n-r}.
\]
Thus the matrix factors as the product of the Vandermonde-type matrix
$
  (q_i^r)_{1\le i\le n,\,0\le r\le n-1}
$
and a coefficient matrix whose $j$-th column is
\[
  \begin{pmatrix}
  x+1+y_j^n\\
  \binom n1y_j^{n-1}\\
  \binom n2y_j^{n-2}\\
  \vdots\\
  \binom n{n-1}y_j
  \end{pmatrix}.
\]
The first factor has nonzero determinant because the $q_i$ are distinct.  The binomial coefficients are nonzero modulo $p$.  After removing them from the last $n-1$ rows and writing the first row as $(x+1)(1,\ldots,1)+(y_1^n,\ldots,y_n^n)$, the two resulting alternants give
\[
  \det
  \begin{pmatrix}
  x+1+y_1^n & \cdots & x+1+y_n^n\\
  y_1^{n-1} & \cdots & y_n^{n-1}\\
  \vdots & & \vdots\\
  y_1 & \cdots & y_n
  \end{pmatrix}
  =\pm\prod_{1\le r<s\le n}(y_s-y_r)
  \left(x+1+(-1)^{n-1}\prod_{j=1}^n y_j\right).
\]
The Vandermonde product is nonzero because the $y_j$ are distinct.  This proves the lemma.
\end{proof}

Let
$
  c\equiv n!\pmod p.
$
Wilson's theorem gives
\begin{equation}\label{eq:wilson}
  c^2\equiv (-1)^{n+1}\pmod p.
\end{equation}
This follows from
\[
  -1\equiv (p-1)!
  \equiv n!\,(p-1)(p-2)\cdots(p-n)
  \equiv (-1)^n c^2\pmod p.
\]

Applying \cref{lem:vandermonde} with $y_j=-aj$ gives
\begin{equation}\label{eq:Da-congruence}
  D_a(x)\equiv C_a\bigl(x+1-\chi(a)c\bigr)\pmod p,
\end{equation}
where $C_a\in\Fp^\times$, since
\[
  \prod_{j=1}^n(-aj)=(-a)^n n!
  \equiv (-1)^n\chi(a)c\pmod p.
\]

If $p\equiv 1\pmod 4$, then $n$ is even and \eqref{eq:wilson} gives $c^2\equiv -1\pmod p$.  Hence $c\ne\pm1$, and therefore
\[
  1-\chi(a)c\ne0\pmod p.
\]
By \eqref{eq:Da-congruence}, $D_a(0)\not\equiv0\pmod p$, so $D_a(0)\ne0$ as an integer.

If $p\equiv 3\pmod 4$, then $n$ is odd and \eqref{eq:wilson} gives $c^2\equiv1\pmod p$.  Thus $c=\pm1$ in $\Fp$, and since $\chi(-1)=-1$ we have
\begin{equation}\label{eq:chianc}
  \chi(a n!)=\chi(a)c.
\end{equation}
Thus \eqref{eq:Da-congruence} gives nonvanishing when $\chi(a n!)=-1$.  The remaining case is supplied by the Pfaffian argument below.

\section{The half-system Pfaffian theorem}\label{sec:half-system}

Assume throughout this section that
$
  p\equiv3\pmod4.
$
Let
$
  \QR=\{u^2:u\in\Fpx\}
$
be the set of nonzero quadratic residues.  Since $\chi(-1)=-1$, each pair $\{u,-u\}$ contains exactly one element of $\QR$.

For any sign function
$
  \varepsilon:\QR\to\{\pm1\},
$
define a half-system
\[
  C_\varepsilon=\{\varepsilon_q q:q\in\QR\}\subset\Fpx.
\]
Thus $C_\varepsilon$ contains exactly one element from each pair $\{u,-u\}$.

Index rows and columns by $\QR$ and define
\begin{align*}
  M_\varepsilon(x)
  &=\bigl(x+\chi(r-\varepsilon_q q)\bigr)_{r,q\in\QR},\\
  M_\varepsilon^{(0)}(x)
  &=
  \begin{pmatrix}
    x & (x-\varepsilon_q)_{q\in\QR}\\
    (x+1)_{r\in\QR} & M_\varepsilon(x)
  \end{pmatrix}.
\end{align*}
The second formula is exactly the augmentation by the row and column indexed by $0$, because
\[
  \chi(0-\varepsilon_q q)=-\varepsilon_q,
  \qquad
  \chi(r-0)=1
  \qquad(r,q\in\QR).
\]

Let
\[
  A=\{q\in\QR:\varepsilon_q=1\},
  \qquad m=|A|.
\]
In the following theorem the set $A$ is placed first in the ordering of $\QR$; the displayed Pfaffian squares are independent of this auxiliary ordering.

\begin{theorem}[Half-system Pfaffian theorem]\label{thm:half-system}
There exists a positive integer $\beta_p$ depending only on $p$ such that the following holds for every sign function $\varepsilon:\QR\to\{\pm1\}$.

Let
\[
  B=(\chi(r+q))_{r,q\in\QR},
  \qquad
  P=(\chi(r-q))_{r,q\in\QR},
\]
and put
$
  T=B^{-1}P.
$
After ordering $\QR$ with $A$ first, write
\[
  T=
  \begin{pmatrix}
    S & R\\
    U & V
  \end{pmatrix},
\]
where $S$ is the $m\times m$ principal block indexed by $A$.

\smallskip
\noindent\textup{(i)} If $m$ is even, then
\[
  \det M_\varepsilon(x)
  =\beta_p^2\Pf(S)^2(nx-1)
\]
and
\[
  \det M_\varepsilon^{(0)}(x)
  =-\beta_p^2\Pf(S)^2\bigl(n+(2n+1)x\bigr).
\]
Here $\Pf(S)=1$ if $m=0$.

\smallskip
\noindent\textup{(ii)} If $m$ is odd, define the even skew-symmetric matrix
\[
  S^+=
  \begin{pmatrix}
    0 & \one_m^t\\
    -\one_m & S
  \end{pmatrix}.
\]
Then
\[
  \det M_\varepsilon(x)
  =\beta_p^2\Pf(S^+)^2x
\]
and
\[
  \det M_\varepsilon^{(0)}(x)
  =\beta_p^2\Pf(S^+)^2.
\]
\end{theorem}

We use two lemmas.

\begin{lemma}[The basic group matrices]\label{lem:B-P}
The matrix $B$ is invertible and
$
  \det B=-\beta_p^2
$
for some positive integer $\beta_p$.  Moreover
\[
  B^t=B,
  \qquad
  P^t=-P,
  \qquad
  BP=PB,
\]
and
\[
  B\one=-\one,
  \qquad
  P\one=0.
\]
Consequently $T=B^{-1}P$ is skew-symmetric and $T\one=0$.
\end{lemma}

\begin{proof}
For $r,q\in\QR$,
\[
  \chi(r+q)=\chi(1+qr^{-1}),
  \qquad
  \chi(r-q)=\chi(1-qr^{-1}),
\]
because $r\in\QR$.  Thus $B$ and $P$ are group matrices on the abelian group $\QR$; hence they commute.

The symmetry of $B$ is immediate.  Since $p\equiv3\pmod4$, $\chi(-1)=-1$, so
\[
  \chi(q-r)=-\chi(r-q),
\]
and therefore $P^t=-P$.

For $r\in\QR$,
\[
  \sum_{q\in\QR}\chi(r+q)=\sum_{t\in\QR}\chi(1+t)=-1,
\]
and
\[
  \sum_{q\in\QR}\chi(r-q)=\sum_{t\in\QR}\chi(1-t)=0.
\]
The two evaluations follow from
\[
  \sum_{t\in\QR}\chi(1+t)
  =\frac12\sum_{t\in\Fpx}(1+\chi(t))\chi(1+t)=-1
\]
and
\[
  \sum_{t\in\QR}\chi(1-t)
  =\frac12\sum_{t\in\Fpx}(1+\chi(t))\chi(1-t)=0,
\]
using the standard quadratic sums
\[
  \sum_{t\in\Fp}\chi(t^2+t)=-1,
  \qquad
  \sum_{t\in\Fp}\chi(-t^2+t)=1.
\]
Thus $B\one=-\one$ and $P\one=0$.

It remains to determine the sign class of $\det B$.  Since $B$ is a group matrix on $\QR$, its eigenvalues are
\[
  \lambda_\psi=\sum_{t\in\QR}\chi(1+t)\psi(t),
\]
where $\psi$ runs through the characters of the cyclic group $\QR$.  The trivial character gives $\lambda_1=-1$.  Also
\[
  \chi(1+t^{-1})=\chi(1+t)
  \qquad(t\in\QR),
\]
so $\lambda_\psi=\lambda_{\psi^{-1}}$.  Since $n=|\QR|$ is odd, the only character equal to its inverse is the trivial character.  Hence
\[
  \det B=-\prod_{\{\psi,\psi^{-1}\},\,\psi\ne1}\lambda_\psi^2.
\]
Each $\lambda_\psi$ is an algebraic integer.  The product over the nontrivial inverse-pairs is fixed by every automorphism of the relevant cyclotomic field, since such automorphisms permute the characters of $\QR$.  Hence the product lies in $\mathbb Q$.  It is also an algebraic integer, and therefore a rational integer.  Thus $\det B=-\beta_p^2$ for some integer $\beta_p\ge0$.

To see that $\beta_p\ne0$, apply \cref{lem:vandermonde} with $y_j$ running through the elements of $\QR$.  Since
\[
  \prod_{q\in\QR}q=(n!)^2\equiv1\pmod p
\]
when $p\equiv3\pmod4$, and since $n$ is odd, the factor in \cref{lem:vandermonde} at $x=0$ is $2\ne0$ in $\Fp$.  Thus $\det B\not\equiv0\pmod p$, so $B$ is invertible and $\beta_p>0$.

Since $B^t=B$, $P^t=-P$, and $BP=PB$,
\[
  T^t=(B^{-1}P)^t=P^tB^{-1}=-PB^{-1}=-B^{-1}P=-T.
\]
Also $T\one=B^{-1}P\one=0$.
\end{proof}

\begin{lemma}[A linear algebra lemma]\label{lem:linear}
Let $T$ be an $n\times n$ skew-symmetric matrix over a field of characteristic $0$, and assume $T\one_n=0$.  Let $A$ be a subset of size $m$, put $k=n-m$, and order the indices with $A$ first.  Write
\[
  T=
  \begin{pmatrix}
    S & R\\
    U & V
  \end{pmatrix},
\]
where $S$ is $m\times m$.  Define
\[
  E=
  \begin{pmatrix}
    S & 0\\
    U & I_k
  \end{pmatrix},
  \qquad
  J=\one_n\one_n^t,
\]
and
\[
  \varepsilon=
  \begin{pmatrix}
    \one_m\\
    -\one_k
  \end{pmatrix}.
\]
Set
\[
  H(x)=
  \begin{pmatrix}
    x & x\one_n^t-\varepsilon^t\\
    -(x+1)\one_n & E-xJ
  \end{pmatrix}.
\]
If $m$ is even, then
\[
  \det(E-xJ)=\det(S)(1-nx)
\]
and
\[
  \det H(x)=\det(S)\bigl(n+(2n+1)x\bigr).
\]
If $m$ is odd, define
\[
  \Delta=
  \det
  \begin{pmatrix}
    0 & \one_m^t\\
    -\one_m & S
  \end{pmatrix}.
\]
Then
\[
  \det(E-xJ)=-x\Delta
\]
and
\[
  \det H(x)=-\Delta.
\]
Moreover,
\[
  \det(S)=\Pf(S)^2\quad(m\text{ even}),
  \qquad
  \Delta=\Pf
  \begin{pmatrix}
    0 & \one_m^t\\
    -\one_m & S
  \end{pmatrix}^{\!2}
  \quad(m\text{ odd}).
\]
\end{lemma}

\begin{proof}
Let $e=\one_m$ and $f=\one_k$.  Since $T\one_n=0$ and $T^t=-T$, we have
\begin{equation}\label{eq:fU}
  f^tU=-e^tS.
\end{equation}
Take the Schur complement of $I_k-xff^t$ in $E-xJ$.  Since
\[
  (I_k-xff^t)^{-1}=I_k+\frac{x}{1-kx}ff^t,
\]
we obtain the polynomial identity
\begin{equation}\label{eq:schur-K}
  \det(E-xJ)
  =(1-kx)
  \det\left(S-\frac{x}{1-kx}e(e^tS+e^t)\right).
\end{equation}
For even $m$, assume first that $S$ is invertible.  Then $S^{-1}$ is skew-symmetric, so $e^tS^{-1}e=0$, and \eqref{eq:schur-K} gives
\[
\begin{aligned}
  \det(E-xJ)
  &=(1-kx)\det(S)
  \left(1-\frac{x}{1-kx}(e^tS+e^t)S^{-1}e\right)\\
  &=(1-kx)\det(S)
  \left(1-\frac{mx}{1-kx}\right)\\
  &=\det(S)(1-nx).
\end{aligned}
\]
Both sides are polynomial in the entries of $S$, so the identity holds for all even skew-symmetric $S$.

For odd $m$, $\det S=0$.  Using \eqref{eq:schur-K} and the adjugate formula for a rank-one perturbation,
\[
  \det(S-\alpha e(e^tS+e^t))
  =-\alpha(e^tS+e^t)\adj(S)e,
\]
where $\alpha=x/(1-kx)$.  Since $S\adj(S)=0$, this becomes
\[
  \det(E-xJ)=-x\,e^t\adj(S)e.
\]
Finally
\[
  e^t\adj(S)e
  =\det
  \begin{pmatrix}
    0 & e^t\\
    -e & S
  \end{pmatrix}
  =\Delta,
\]
which proves the asserted formula for $\det(E-xJ)$.

For $H(x)$, note that
\[
  H(x)=H(0)+x
  \begin{pmatrix}
    1\\
    -\one_n
  \end{pmatrix}
  \begin{pmatrix}
    1 & \one_n^t
  \end{pmatrix}.
\]
Thus $\det H(x)$ is affine in $x$.  At $x=-1$ the lower-left block of $H(-1)$ is zero, so
\[
  \det H(-1)=-\det(E+J).
\]
Using the formula for $\det(E-xJ)$ at $x=-1$ gives
\[
  \det H(-1)=
  \begin{cases}
    -(n+1)\det S, & m\text{ even},\\
    -\Delta, & m\text{ odd}.
  \end{cases}
\]
At $x=0$, taking the Schur complement of the lower-right identity block $I_k$ gives
\begin{equation}\label{eq:H0}
  \det H(0)
  =\det
  \begin{pmatrix}
    k & e^t(S-I_m)\\
    -e & S
  \end{pmatrix}.
\end{equation}
If $m$ is even and $S$ is invertible, \eqref{eq:H0} equals
\[
  \det(S)\left(k+e^t(S-I_m)S^{-1}e\right)
  =\det(S)(k+m)=n\det(S),
\]
and the general even case follows because both sides are polynomial in the entries of $S$.  If $m$ is odd, then the term containing $k\det S$ vanishes, the contribution from the row $e^tS$ vanishes because $S\adj(S)=0$, and the remaining contribution is $-\Delta$.  Hence
\[
  \det H(0)=
  \begin{cases}
    n\det S, & m\text{ even},\\
    -\Delta, & m\text{ odd}.
  \end{cases}
\]
Since $\det H(x)$ is affine in $x$, the two values at $x=0$ and $x=-1$ determine it, giving
\[
  \det H(x)=\det(S)\bigl(n+(2n+1)x\bigr)
\]
for even $m$, and
\[
  \det H(x)=-\Delta
\]
for odd $m$.

The final Pfaffian identities are the standard identities $\det W=\Pf(W)^2$ for even skew-symmetric matrices $W$; see, for example, \cite{Knuth}.
\end{proof}

\begin{proof}[Proof of \cref{thm:half-system}]
Let $N_\varepsilon$ be the matrix without the $xJ$ part:
\[
  N_\varepsilon=(\chi(r-\varepsilon_q q))_{r,q\in\QR}.
\]
If $q\in A$, the $q$-column of $N_\varepsilon$ is the $q$-column of $P$; if $q\notin A$, it is the $q$-column of $B$.  Since $T=B^{-1}P$, we have
\[
  N_\varepsilon=BE,
\]
with $E$ as in \cref{lem:linear}.  Also $B^{-1}\one=-\one$ by \cref{lem:B-P}.  Therefore
\[
  M_\varepsilon(x)=N_\varepsilon+xJ
  =B(E-xJ),
\]
and hence
\[
  \det M_\varepsilon(x)=\det B\,\det(E-xJ).
\]
Using \cref{lem:B-P} and \cref{lem:linear} gives the asserted formulas for $\det M_\varepsilon(x)$.

For the augmented determinant, multiply the lower $n$ rows by $B^{-1}$.  Since this operation multiplies the determinant by $\det(B)^{-1}$, we get
\[
  \det M_\varepsilon^{(0)}(x)
  =\det B\,
  \det
  \begin{pmatrix}
    x & x\one^t-\varepsilon^t\\
    -(x+1)\one & E-xJ
  \end{pmatrix}.
\]
By \cref{lem:linear}, the right-hand determinant is $\det H(x)$.  The result follows from \cref{lem:B-P,lem:linear}.
\end{proof}

\section{Dilated half-systems and the sign of the column permutation}\label{sec:dilated}

We now return to the determinants $D_a(x)$ and $D_a^{(0)}(x)$.  Assume first that $p\equiv3\pmod4$.  The set
\[
  C_a=\{aj:1\le j\le n\}\subset\Fpx
\]
contains exactly one element from each pair $\{u,-u\}$, because $\{1,2,\dots,n\}$ itself has this property.  Therefore there is a unique sign function $\varepsilon_a:\QR\to\{\pm1\}$ such that
\[
  C_a=\{\varepsilon_{a,q}q:q\in\QR\}.
\]
Let
\[
  A_a=\{q\in\QR:\varepsilon_{a,q}=1\},
  \qquad m_a=|A_a|.
\]

\begin{lemma}[Parity of $m_a$]\label{lem:parity}
Assume $p\equiv3\pmod4$.  Then
\[
  m_a\text{ is odd}
  \quad\Longleftrightarrow\quad
  \chi(a n!)=1.
\]
\end{lemma}

\begin{proof}
Taking the product of the elements of $C_a$ gives
\[
  \prod_{j=1}^n aj=a^n n!\equiv\chi(a)c\pmod p.
\]
On the other hand,
\[
  \prod_{q\in\QR}\varepsilon_{a,q}q
  =(-1)^{n-m_a}\prod_{q\in\QR}q
  =(-1)^{n-m_a}(n!)^2.
\]
When $p\equiv3\pmod4$, \eqref{eq:wilson} gives $(n!)^2\equiv1\pmod p$, and \eqref{eq:chianc} gives $\chi(a n!)=\chi(a)c$.  Hence
\[
  \chi(a n!)=(-1)^{n-m_a}.
\]
Since $n$ is odd, the right-hand side is $1$ exactly when $m_a$ is odd.
\end{proof}

\begin{lemma}[Column sign]\label{lem:column-sign}
Assume $p\equiv3\pmod4$, and order $\QR$ as $1^2,2^2,\dots,n^2$.  Let $\tau_a$ be the sign of the permutation that reorders the columns $j=1,\dots,n$ of $D_a(x)$ into the order indexed by $\QR$ through
\[
  aj=\varepsilon_{a,q}q.
\]
Then
\[
  \tau_a=s_p=(-1)^{\lfloor(p+1)/8\rfloor}.
\]
\end{lemma}

\begin{proof}
For each $j\in\{1,\dots,n\}$, the corresponding element of $\QR$ is
\[
  q(j)=\chi(aj)aj=\chi(a)a\,\chi(j)j.
\]
Multiplication by the fixed element $\chi(a)a\in\QR$ has sign $+1$ on $\QR$: after identifying $\QR$ with the quotient group
\[
  G=\Fpx/\{\pm1\},
\]
it is a translation in the group $G$, whose order $n$ is odd; all translation cycles have odd length and hence even sign.

It remains to compute the sign of $j\mapsto\chi(j)j$.  Let $\pi$ be the permutation of $\{1,\dots,n\}$ determined by
\[
  \pi(j)^2\equiv\chi(j)j\pmod p.
\]
In the quotient group $G$, this means
\[
  [\pi(j)]^2=[j].
\]
Since $|G|=n$ is odd, the inverse of the squaring map on $G$ is the power map
\[
  [u]\longmapsto [u]^r,
  \qquad r=\frac{n+1}{2}=\frac{p+1}{4}.
\]
Thus $\pi$ is conjugate to the power permutation $g\mapsto g^r$ of the cyclic group $G$.  For $n=1$ the sign is $1$, and we use the convention $(a/1)=1$.  For odd $n>1$, Zolotarev's lemma in its Jacobi-symbol form gives the sign of multiplication by $r$ on a cyclic group of order $n$; see \cite[Chapter~3]{IrelandRosen}.  Hence
\[
  \sgn(\pi)=\left(\frac{r}{n}\right).
\]
Since $2r\equiv1\pmod n$,
\[
  \left(\frac{r}{n}\right)=\left(\frac{2}{n}\right)
  =(-1)^{(n^2-1)/8}.
\]
If $p=8h+3$, then $n=4h+1$ and $(n^2-1)/8\equiv h\pmod2$.  If $p=8h+7$, then $n=4h+3$ and $(n^2-1)/8\equiv h+1\pmod2$.  In both cases
\[
  (-1)^{(n^2-1)/8}=(-1)^{\lfloor(p+1)/8\rfloor}.
\]
This proves the lemma.
\end{proof}

\section{Proof of the main theorems}\label{sec:proofs}

\begin{proof}[Proof of \cref{thm:square}]
Assume $p\equiv3\pmod4$.  Reordering the nonzero columns of $D_a(x)$ and $D_a^{(0)}(x)$ as in \cref{lem:column-sign} changes both determinants by the same sign $s_p$.  In the new order they are precisely the half-system determinants attached to $\varepsilon_a$.  Therefore
\begin{align*}
  s_pD_a(x)&=\det M_{\varepsilon_a}(x),\\
  s_pD_a^{(0)}(x)&=\det M_{\varepsilon_a}^{(0)}(x).
\end{align*}
We shall use the following elementary fact: if $R=u/v\in\mathbb Q$ with $(u,v)=1$ and $R^2\in\mathbb Z$, then $v^2\mid u^2$, hence $v=1$ and $R\in\mathbb Z$.

If $\chi(a n!)=1$, then $m_a$ is odd by \cref{lem:parity}.  Applying \cref{thm:half-system} gives
\[
  s_pD_a(x)=\beta_p^2\Pf(S_a^+)^2x
\]
and
\[
  s_pD_a^{(0)}(x)=\beta_p^2\Pf(S_a^+)^2.
\]
The factor $\beta_p\Pf(S_a^+)$ is rational because $T=B^{-1}P$ has rational entries.  It is nonzero: by \eqref{eq:Da-congruence}, the polynomial $D_a(x)$ is congruent modulo $p$ to $C_ax$ with $C_a\ne0$.  Hence
\[
  \rho_{p,a}=\bigl|\beta_p\Pf(S_a^+)\bigr|>0
\]
has the desired property over $\mathbb Q$.  Since $s_pD_a(x)/x$ has integral coefficients and is constant, $\rho_{p,a}^2\in\mathbb Z$; hence $\rho_{p,a}\in\mathbb Z$.

If $\chi(a n!)=-1$, then $m_a$ is even.  Applying \cref{thm:half-system} gives
\[
  s_pD_a(x)=\beta_p^2\Pf(S_a)^2(nx-1)
\]
and
\[
  s_pD_a^{(0)}(x)
  =-\beta_p^2\Pf(S_a)^2\bigl(n+(2n+1)x\bigr).
\]
Again the Pfaffian is rational.  It is nonzero because \eqref{eq:Da-congruence} gives
\[
  D_a(0)\equiv 2C_a\not\equiv0\pmod p.
\]
Put
\[
  \sigma_{p,a}=\bigl|\beta_p\Pf(S_a)\bigr|>0
\]
Then $\sigma_{p,a}$ works over $\mathbb Q$.  Evaluating at $x=0$ gives $\sigma_{p,a}^2=-s_pD_a(0)\in\mathbb Z$, so $\sigma_{p,a}\in\mathbb Z$.
\end{proof}

\begin{proof}[Proof of \cref{thm:vanishing}]
If $p\equiv1\pmod4$, then the argument following \eqref{eq:Da-congruence} shows that $D_a(0)\ne0$.

Assume now that $p\equiv3\pmod4$.  If $\chi(a n!)=1$, then \cref{thm:square} gives $D_a(x)=s_p\rho_{p,a}^{\,2}x$, so $D_a(0)=0$.  If $\chi(a n!)=-1$, then \eqref{eq:Da-congruence} gives $D_a(0)\not\equiv0\pmod p$, hence $D_a(0)\ne0$.  This proves the criterion.
\end{proof}

\begin{proof}[Proof of \cref{thm:sun}]
Take $a=n!$.  Then
\[
  \chi(a n!)=\chi((n!)^2)=1.
\]
The theorem is exactly \cref{thm:vanishing,thm:square} in this special case.
\end{proof}

\bigskip
\noindent\textsc{Hong-Ge Chen},
School of Mathematics and Statistics,  Central China Normal University, Wuhan
430079, China.

\noindent\emph{Email address}: \href{mailto:hongge\_chen@whu.edu.cn}{\texttt{hongge\_chen@whu.edu.cn}}

\medskip
\noindent\textsc{Fei Liu},
Department of Mathematics, Run Run Shaw Building, The University of Hong Kong,
Hong Kong.

\noindent\emph{Email address}: \href{mailto:liufei54@pku.edu.cn}{\texttt{liufei54@pku.edu.cn}}


\begin{thebibliography}{9}

\bibitem{Sun}
Z.-W. Sun,
\emph{On some determinants with Legendre symbol entries},
Finite Fields Appl. \textbf{56} (2019), 285--307.

\bibitem{GrinbergSunZhao}
D. Grinberg, Z.-W. Sun and L. Zhao,
\emph{Proof of three conjectures on determinants related to quadratic residues},
arXiv:2007.06453.

\bibitem{IrelandRosen}
K. Ireland and M. Rosen,
\emph{A Classical Introduction to Modern Number Theory},
2nd ed.,
Graduate Texts in Mathematics, vol. 84,
Springer, 1990.

\bibitem{Wu}
H.-L. Wu,
\emph{Determinants concerning Legendre symbols},
arXiv:2012.00502.

\bibitem{WangWu}
L.-Y. Wang and H.-L. Wu,
\emph{On the cyclotomic field $\mathbb Q(e^{2\pi{\bf i}/p})$ and Zhi-Wei Sun's conjecture on $\det M_p$},
Finite Fields Appl. \textbf{101} (2025), Article 102533; arXiv:2401.05853.

\bibitem{RenLuo}
C.-K. Ren and X.-Q. Luo,
\emph{On certain determinants and the square root of some determinants involving Legendre symbols},
arXiv:2407.04556.

\bibitem{RenSun}
C.-K. Ren and Z.-W. Sun,
\emph{Evaluation of a determinant involving Legendre symbols},
arXiv:2507.18589.

\bibitem{YangZhang}
Y. Yang and Y. Zhang,
\emph{Two determinant evaluations in Sun's conjectures involving Legendre symbols},
arXiv:2605.19517.

\bibitem{Knuth}
D. E. Knuth,
\emph{Overlapping Pfaffians},
Electron. J. Combin. \textbf{3} (1996), no. 2, Research Paper 5.

\end{thebibliography}
\end{document}